\documentclass{amsart}
\usepackage{amsmath}
  \usepackage{graphics} 
\usepackage{graphicx}  
 \usepackage[colorlinks=true]{hyperref}
\usepackage{psfrag}
\newtheorem{theorem}{Theorem}[section]

\newtheorem{lemma}[theorem]{Lemma}

\theoremstyle{definition}

\usepackage{dsfont}   
\usepackage{color}

\newcommand{\nn}{\nonumber}

\newcommand{\beq}{\begin{equation}}
\newcommand{\beqa}{\begin{eqnarray}}
\newcommand{\eeq}{\end{equation}}
\newcommand{\eeqa}{\end{eqnarray}}
\newcommand{\jri}{j \rightarrow \infty}

\newcommand{\Nb}{{\mathbb N}}

\newcommand{\Rb}{{\mathbb R}}
\newcommand{\sgn}{\operatorname{sgn}} 
\newcommand{\T}{\tau}

\renewcommand{\t}{\tau}

\title[On the convergence to critical scaling profiles]{On the convergence to critical scaling profiles 
in submonolayer deposition models}

\author[F.P. da Costa, J.T. Pinto and R. Sasportes]{}
\email{fcosta@uab.pt}
\email{jpinto@tecnico.ulisboa.pt}
\email{rafael@uab.pt}
\thanks{Part of this work was done while FPdC was staying at the Department of Mathematics of the
Faculty of Natural Sciences of the National University of Laos, Vientiane, Laos, with the support of an 
Erasmus Mundus Mobility with Asia grant EMMA ID 2601. The hospitality of NUOL is gratefully acknowledged.}

\thanks{Partially funded by FCT/Portugal through project RD0447/CAMGSD/2015.}
\thanks{$^*$ Corresponding author: F.P. da Costa}
\dedicatory{Dedicated to the memory of Jack Carr}

\begin{document}
\maketitle

%
%
%
\centerline{\scshape Fernando P. da Costa$^*$}
\medskip
{\footnotesize
 \centerline{Departamento de Ci\^encias e Tecnologia, Universidade Aberta, Lisboa, Portugal}
 \centerline{and}
   \centerline{Centro de An\'alise Matem\'atica, Geometria e Sistemas Din\^amicos, }
   \centerline{Instituto Superior T\'ecnico,
Universidade de Lisboa, Lisboa, Portugal}
} 

\medskip

\centerline{\scshape Jo\~ao T. Pinto}
\medskip
{\footnotesize
   \centerline{Departamento de Matem\'atica, Instituto Superior T\'ecnico,
Universidade de Lisboa, Lisboa, Portugal}
\centerline{and}
   \centerline{Centro de An\'alise Matem\'atica, Geometria e Sistemas Din\^amicos, }
   \centerline{Instituto Superior T\'ecnico,
Universidade de Lisboa, Lisboa, Portugal}
}

\medskip

\centerline{\scshape Rafael Sasportes}
\medskip
{\footnotesize
 \centerline{Departamento de Ci\^encias e Tecnologia, Universidade Aberta, Lisboa, Portugal}
 \centerline{and}
   \centerline{Centro de An\'alise Matem\'atica, Geometria e Sistemas Din\^amicos, }
   \centerline{Instituto Superior T\'ecnico,
Universidade de Lisboa, Lisboa, Portugal}
} 
\bigskip
 \centerline{July 8, 2017}

\begin{abstract}
In this work we study the rate of convergence to similarity profiles in a mean field model
for the deposition of a submonolayer of atoms in a crystal facet, when there is a critical minimal 
size $n\geq 2$ for the stability of the formed clusters. 
The work complements recently published related results by the same authors
in which the rate of convergence 
was studied outside of a critical direction $x=\tau$ in the cluster size $x$ vs. time $\t$ plane. 
In this paper we consider a
different similarity variable, $\xi := (x-\t)/\sqrt{\t}$, corresponding to an inner expansion of that critical 
direction, 
and prove the convergence of solutions to a similarity profile $\Phi_{2,n}(\xi)$ when 
$x, \t\to +\infty$ with $\xi$ fixed, as
well as the rate at which the limit is approached.  
\end{abstract}







\section{Introduction}\label{sec:intr}

The deposition of a monolayer of atoms on top of a single crystal facet is a process 
of clear technological importance, and its theoretical understanding is a theme of 
current scientific interest (see, e.g., \cite{Mul09}).

\medskip

In the early stages of the deposition process the new monolayer is still far from being completed
and only islands, or $j$-clusters, made up of a number $j$ of adatoms exist, 
which, being far apart from each other on the crystal facet, 
do not interact among them. 

\medskip

One of the mathematical descriptions of this stage of the deposition process consists in a coagulation type 
differential equation \cite{crw,costin} obtained as follows: denote by $c_j=c_j(t)$ the 
concentration of $j$-clusters at time $t$ on the
crystal facet. The surface is hit at a constant rate $\alpha>0$ by 
1-clusters (also called \emph{monomers}). Consider the
non interacting assumption above. Then,  
the only reactions of the clusters on the crystal surface are those in which 
a monomer takes part, namely $(1)+(j)\rightarrow (j+1)$, with rate coefficients $a_{1,j}\geq 0$, 
for $j\in\{1,2,\ldots\}$. Thus, the differential equations governing the dynamics of the
population of clusters $\{c_j(t)\}$ are a Smoluchowski coagulation system with 
Becker-D\"oring like coagulation kernel (see, e.g., \cite{c15,w06}).

\medskip

An additional assumption relevant to some applications is the existence of a critical 
cluster size $n$ below which the clusters on the crystal facet are not stable and do not 
exist in any significant amount in the time scale
of the coagulation reactions. Mathematically, this behaviour can be described as follows, \cite{costin}:
clusters of size larger than $1$ and smaller than or equal to $n-1$ cannot exist, and thus the smaller
cluster that is not a monomer has size $n$ and is formed when $n$ monomers come together and react
into an $n$-cluster. These ``multiple collisions'' can be though of as follows \cite{cps}: if, as in Monte Carlo
simulations, we consider the monomers sitting in the vertices of a lattice, 
and if we have $n-1$ monomers surrounding an empty site -- as in the four nearest
neighbours in a square lattice -- which is suddenly hit by a monomer, 
we have one of these ``multiple'' collisions and the creation of an 
$n$-cluster in a single $n$ body $1$-cluster reaction.

\medskip

Considering all
coagulation rates time and cluster size independent ($a_{1,j}\equiv 1$, say), the mean-field model 
for the submonolayer deposition with a critical cluster size $n$ is given by the following coagulation
type ordinary differential equation system:

\beq
\left\{
\begin{array}{lcl}
\dot{c}_1 & = & \alpha - nc_1^n - c_1 \displaystyle{\sum_{j=n}^{\infty}c_j}, \\
\dot{c}_n & = & c_1^n-c_1c_n,  \\
\dot{c}_j & = & c_1c_{j-1}-c_1c_j, \;\;\; j \geq n+1. 
\end{array}
\right.\label{system}
\eeq

The convergence of solutions of \eqref{system} to a scaling profile was studied in \cite{costin}. 
Using the new time scale 
\begin{equation}
\t(t):=\int_{t_0}^tc_1(s)ds\label{newtime}
\end{equation}
it was proved that solutions of \eqref{system} in this time scale, 
\begin{equation}
\widetilde{c}_j(\t):=c_j(t(\t)), \label{solnewtime}
\end{equation}
satisfy the following scaling behaviour:
\[
\displaystyle{
\lim_{\scriptsize{\begin{array}{c}j,\,\t \rightarrow +\infty\\ \eta=j/\t\; \mbox{\rm  fixed}\\ 
\eta \neq 1\end{array}}}\!\!
\left(\frac{n\t}{\alpha}\right)^{(n-1)/n}\!\!
\widetilde{c}_j(\t) = \Phi_{1,n}(\eta) :=
\left\{\begin{array}{ll} \left(1-\eta\right)^{-(n-1)/n},& \mbox{\rm  if $\;0<\eta<1$}\\
0, & \mbox{\rm  if $\;\eta>1$}.
\end{array}\right.
}
\]
\noindent
This extended the result in \cite{crw}, valid for the absence of critical 
cluster sizes (i.e., when $n=2$).
Exploiting techniques based on center manifold theory used in \cite{crw}, we recently
proved in \cite{cps} results  about the rate of convergence to this scaling behaviour.

\medskip

Along the direction $\eta:= j/\t = 1$  a different scaling variable
is needed, which was identified 
in \cite{crw} as
$\xi := (j-\t)/\sqrt{\t} \in \Rb$. Observing solutions of \eqref{system} along $\xi = \text{const.}$
corresponds to look at them with $\eta\to 1$ as $\t\to\infty$ and so, in a certain sense, the new
scaling variable $\xi$ blows up to the whole real line the singular behaviour identified at the
critical value $\eta=1$.

%
%
\begin{figure}[h]\label{fig:sslines}
	\begin{center}
		\psfrag{j}{$j$}
		\psfrag{tau}{$\tau$}
		\psfrag{n=2}{$n=2$}
		\psfrag{n=3}{$n=3$}
		\psfrag{n=6}{$n=6$}
		\psfrag{n=12}{$n=12$}
		\psfrag{xi}{$\xi$}
		\psfrag{-6}{$-6$}
		\psfrag{-4}{$-4$}
		\psfrag{-2}{$-2$}
		\psfrag{0}{$0$}
		\psfrag{2}{$2$}
		\psfrag{4}{$4$}
		\psfrag{0.5}{$0.5$}
		\psfrag{1.0}{$1.0$}
		\psfrag{1.5}{$1.5$}
		\psfrag{2.0}{$2.0$}
		\psfrag{1.0}{$1.0$}
		\psfrag{1.2}{$1.2$}
		\includegraphics[scale=1]{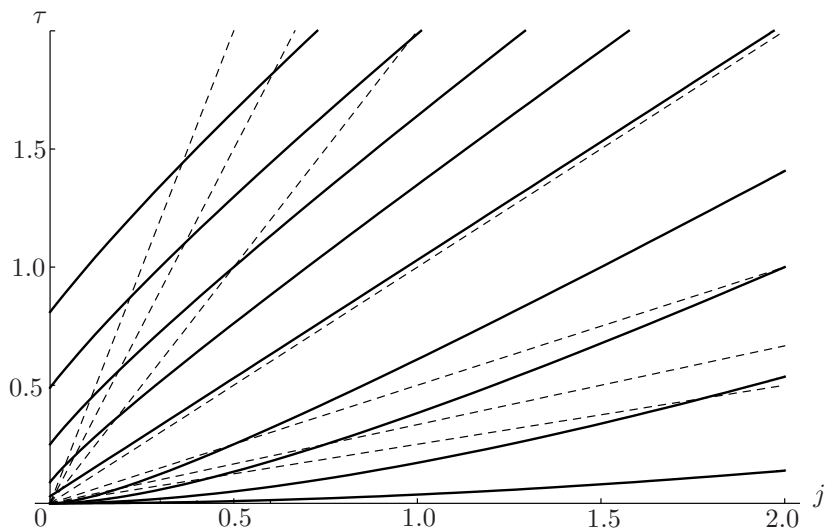}
	\end{center}
	\caption{Lines with $\xi = \text{constant}$ (full), and with $\eta = \text{constant}$ (dashed) in the $(j,\tau)$ plane. 
	The values used for these parameters are the following, in
	counterclockwise direction: $\xi = 5.0, 2.0, 1.0, 0.5, 0.0, -0.3, -0.5, -0.7, -0.9$ and 
	$\eta = 1/4, 1/3, 1/2, 
	1, 2, 3, 4$. }
\end{figure}
%
%

In \cite{crw} the convergence to a similarity profile in this new similarity variable was proved for
the case with no critical size clusters ($n=2$). In this paper we extend that result to 
general $n\geq 2$ by proving that
$$
\lim_{\substack{j,\,\tau \rightarrow +\infty\\ \xi=(j-\tau)/\sqrt{\tau}\; {\rm fixed}\\ \xi\in\mathbb{R}}}
\frac{\sqrt{2\pi}}{\alpha}\left(\frac{\alpha}{n}\right)^{\frac{1}{n}}\tau^{\frac{n-1}{2n}}
\widetilde{c_j}(\tau)= \Phi_{2,n}(\xi)
$$
where
$$
\Phi_{2,n}(\xi):=e^{-\xi^2/2}\int_{0}^{+\infty}
\exp\left(-\xi w^n-\frac{w^{2n}}{2}\right)\,dw\,.
$$

%
%
\begin{figure}[ht!]\label{fig:phi2n}
\begin{center}
	\psfrag{n=2}{$n=2$}
	\psfrag{n=3}{$n=3$}
	\psfrag{n=6}{$n=6$}
	\psfrag{n=12}{$n=12$}
	\psfrag{xi}{$\xi$}
	\psfrag{-6}{$-6$}
	\psfrag{-4}{$-4$}
	\psfrag{-2}{$-2$}
	\psfrag{0}{$0$}
	\psfrag{2}{$2$}
	\psfrag{4}{$4$}
	\psfrag{0.2}{$0.2$}
	\psfrag{0.4}{$0.4$}
	\psfrag{0.6}{$0.6$}
	\psfrag{0.8}{$0.8$}
	\psfrag{1.0}{$1.0$}
	\psfrag{1.2}{$1.2$}
	\includegraphics[scale=0.75]{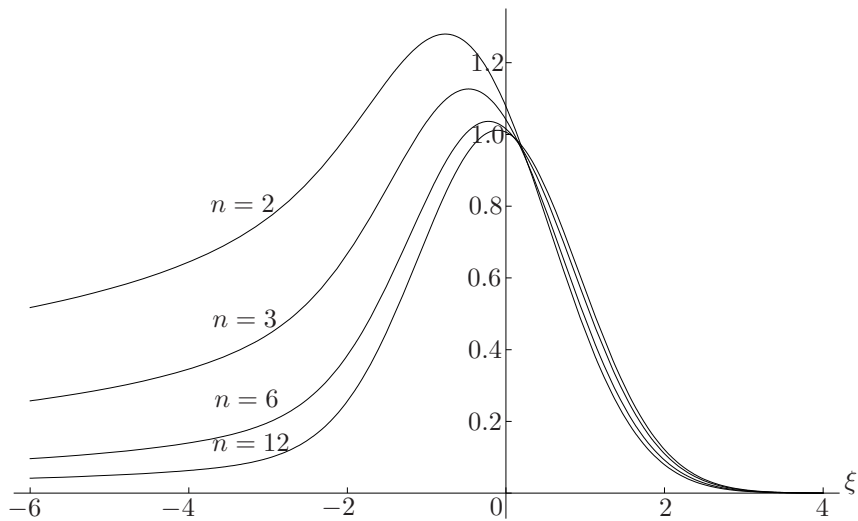}
\end{center}
\caption{Graph of the similarity profile $\Phi_{2,n}(\xi)$ for different values of $n$.}
\end{figure}
%
%

The rate of convergence to similarity profiles in several types of 
coagulating systems has recently attracted some attention (see, e.g., \cite{mp,sri}). 
As stated above, for the deposition model \eqref{system} we recently studied in \cite{cps}
the rate of convergence to the profile $\Phi_{1,n}$ in the 
similarity variable $\eta$. In the present paper we complement that work by studying the rate of convergence
to $\Phi_{2,n}$ in the variable $\xi$, thus completing the study in \cite{cps,crw,costin}.

\medskip

Our main result is the following:

\begin{theorem}\label{teo:main} 
Let $n\geq 2$. For every solution $(c_j)$ of \eqref{system} with initial condition $(c_{j0})$
consider the change of variables $(t, c_j)\mapsto (\tau,\widetilde{c}_j)$ introduced 
in \eqref{newtime} and \eqref{solnewtime}. 
If there exists $\rho_1, \rho_2>0, \mu > 1-\frac{1}{n}$ such that
for all $j\geq n$, $j^{\mu}c_j(0)\in [\rho_1, \rho_2]$ then,
as $j,\,\tau \rightarrow +\infty,$ with $\xi=\frac{j-\tau}{\sqrt{\tau}}$ fixed, the following holds:
\begin{eqnarray}
\left|\frac{\sqrt{2\pi}}{\alpha}\left(\frac{\alpha}{n}\right)^{\frac{1}{n}}\tau^{\frac{n-1}{2n}}
\widetilde{c}_j(\tau) - \Phi_{2,n}(\xi)\right| & \leq & 
\begin{cases}
O(j^{-\tfrac{1}{2n}+\frac{1-\mu}{2}}) & \text{if $1-\frac{1}{n}<\mu<1$}\\
O(j^{-\tfrac{1}{2n}}\log j) & \text{if $\mu=1$}\\
O(j^{-\tfrac{1}{2n}}) & \text{if $\mu>1$.}
\end{cases} \label{maintheo}
\end{eqnarray}
\end{theorem}

\medskip

As a consequence of the proof of this theorem, we obtain also the following result valid for an 
important case in applications: the deposition of monomers on top of
a crystal facet having no $j$-cluster at the initial time $t=0$, i.e., with initial conditions $c_j(0)=0$
for all $j\geq n$:

\begin{theorem}\label{teo:monomer} 
With the notation used in Theorem~\ref{teo:main} consider initial data with $c_{j0}=0$ for all
$j\geq n$. Then,
as $j,\,\tau \rightarrow +\infty,$ with $\xi=\frac{j-\tau}{\sqrt{\tau}}$ fixed, it holds:
\begin{eqnarray}
\left|\frac{\sqrt{2\pi}}{\alpha}\left(\frac{\alpha}{n}\right)^{\frac{1}{n}}\tau^{\frac{n-1}{2n}}
\widetilde{c}_j(\tau) - \Phi_{2,n}(\xi)\right| & \leq & Cj^{-\tfrac{1}{2}}\log j. 
 \label{theomonomer}
\end{eqnarray}
\end{theorem}

\medskip


\section{Preliminaries}\label{sec:prelim}

We briefly present our approach to the study of \eqref{system}, which 
follows the one used in \cite{cps,crw,costin} and consists of 
the exploration of the following two observations:

\begin{enumerate}
\item First note that the equation for $c_1$ depends only on $c_1$ and on the ``bulk'' quantity $y(t):=\sum_{j=n}^{\infty}c_j(t)$, which 
(formally) satisfies the differential equation $\dot{y}=c_1^n$. Thus, the definition of this bulk variable allow
us to decouple the resulting infinite dimensional system into a closed two-dimensional system for the 
monomer--bulk variables $(c_1, y)$,
from which we get all the needed information about the behaviour of $c_1$.
\item Secondly, the remaining equations for \(c_j\), with \(j\geq n\), depend only on those same variables $c_j$, and on $c_1$.
However, the way they depend on $c_1$ is such that, by the change of time variable \eqref{newtime}, 
the system is transformed
into a linear lower triangular infinite system of ordinary differential equations
for  $\widetilde{c}_j(\tau)=c_j(t(\tau))$ , which can be recursively solved,
in terms of \(\widetilde{c}_1\), using the variation
of constants formula.
\end{enumerate}

\medskip

In fact,
writing the second and third equations of (\ref{system}) in the form
\begin{equation}
  \left\{
  \begin{array}{l}
  {\dot{c}_{n}}=c_1({c_1}^{n-1}-c_n)\\
  \dot{c}_{j}=c_1(c_{j-1}-c_j), \;\;\; j \geq n+1,
  \end{array}
  \right.\label{eq:bixc}
\end{equation}
and introducing the new time scale \eqref{newtime}
along with scaled variables \eqref{solnewtime} system  (\ref{eq:bixc}) reads 
\begin{equation}
  \left\{
  \begin{array}{l}
   {\widetilde{c}_n}{'}=\widetilde{c_1}^{n-1}-\widetilde{c}_n\\
 {\widetilde{c}_j}{'}=\widetilde{c}_{j-1}-\widetilde{c}_j,\;\;\; j \geq n+1,
  \end{array}
  \right.\label{eq:bixctil}
 \end{equation}
where \((\cdot)'=d/d\tau\).
The equation for \( {\widetilde{c}_n}\) can  be readily solved in terms of \(\widetilde{c_1}\) resulting in
\[
{\widetilde{c}_n}(\T) = e^{-\T}{\widetilde{c}_n}(0)+\int_0^\T (\widetilde{c_1}(\T-s))^{n-1}e^{-s}ds.
\]
We can then use this solution to solve recursively for \(j\geq n+1\). 
The following expression for \(\widetilde{c}_j(\T)\) is thus obtained :
\beq
\widetilde{c}_j(\tau) = \mathfrak{I}_1(j,\tau) + \mathfrak{I}_2(j,\tau), \label{solc}
\eeq
where
\beq
\mathfrak{I}_1(j,\tau) := 
e^{-\tau}\sum_{k=n}^{j}\frac{\tau^{j-k}}{(j-k)!}c_k(0),\label{I1}
\eeq
and
\beq
\mathfrak{I}_2(j,\tau) := 
\frac{1}{(n-j)!}\int_{0}^{\tau}(\widetilde{c_1}(\tau-s))^{n-1}s^{j-n}e^{-s}\,ds.\label{I2}
\eeq

\medskip

The information about $\widetilde{c}_1$ needed to study \eqref{I2} is 
extracted from the two-dimensional system for $(c_1,y)$
referred to in  (1) above, which is
\begin{equation}
  \left\{
  \begin{array}{l}
 {\dot{x}}=\alpha -nx^n-xy\\
 {\dot{y}}=x^n,
  \end{array}
  \right.\label{eq:xy}
\end{equation}
where \( x(t):= c_{1}(t)\) and \( y(t):=\sum_{j=n}^{\infty}c_{j}(t)\).
This was done in \cite{cps,costin} using asymptotic and center manifold methods.
The result that we need below is the following:

\begin{theorem}\label{teo:longtime} \cite[Theorem 6]{cps}\mbox{}
 With $\T$ and $\widetilde{c}_1(\T)$ as before, the following holds: 
 \begin{equation}
  \left(\frac{n\tau}{\alpha}\right)^{(n-1)/n}\left(\widetilde{c}_1(\T)\right)^{n-1} = 1 + 
  (n-1)\left(1-\textstyle{\frac{1}{n}}\right)\frac{\log\T}{\T} + o\Bigl(\textstyle{\frac{\log\T}{\T}}\Bigr), 
  \ \mbox{\rm as $\T\to\infty$}. \label{longtime}
 \end{equation}
\end{theorem}

\medskip

In the following two sections we prove Theorem~\ref{teo:main} by studying separately the
contributions of the two terms \eqref{I1} and \eqref{I2} to the limit \eqref{solc}.
As a consequence of the estimates in section~\ref{sec:rateest} a proof of Theorem~\ref{teo:monomer}
will also be obtained.


\section{Rates of convergence to $\Phi_{2,n}$: contribution of $\mathfrak{I}_1$}\label{sec:rateestinit}

In this section we study the limit
\beq
\lim_{\substack{j,\,\tau \rightarrow +\infty\\ \xi=\frac{j-\tau}{\sqrt{\tau}}\; {\rm fixed}}}
\frac{\sqrt{2\pi}}{\alpha}\left(\frac{\alpha}{n}\right)^{\frac{1}{n}}\tau^{\frac{n-1}{2n}}
\mathfrak{I}_1(j,\tau).\label{criticallimit}
\eeq
Part of the analysis was presented in \cite{crw} for the case $n=2$ but the study of $S_4$
below is new and proves the conjecture left open in section 6.2 of \cite{crw}.
Using the definition of $\xi$ we can write $j=\tau+\xi\sqrt{\tau},$ which, solved for $\tau$ gives
$\tau = j\Delta_j$ where the function $\Delta_j:\Rb\longrightarrow\Rb$ is defined by
\beq
\Delta_j(\xi):=\left(\sqrt{1+\frac{\xi^2}{4j}}\;+\sgn(-\xi)\sqrt{\frac{\xi^2}{4j}}\;\right)^2. \label{deltaj}
\eeq
A number of properties of this function can be easily established, in particular,
\begin{equation}
\Delta_j  \,
\begin{cases}
<1 & \text{\;\;if $\xi>0$}\\ =1 & \text{\;\;if $\xi=0$} \\ >1 & \text{\;\;if $\xi<0$,}
\end{cases} \label{signdelta}
\end{equation}
and 
\begin{equation}
 \Delta_j   \rightarrow 1 \text{\;\;as $\jri$, for all $\xi$.} \label{limitdelta}
\end{equation}
Assuming $c_k(0)\leq \rho/k^{\mu}$ in \eqref{I1}, changing the summation variable to $\ell:=j-k$, writing
$\tau$ in terms of $j$ as indicated, and not taking multiplication constants
into account, the limit (\ref{criticallimit}) can be written as
\beq
\lim_{\jri}
(j \Delta_j)^{\frac{n-1}{2n}}\!e^{-j\Delta_j}\sum_{\ell=0}^{j-n}\frac{(j\Delta_j)^{\ell}}{\ell !(j-\ell)^{\mu}}.\label{jlimit}
\eeq

\medskip

In order to
evaluate (\ref{jlimit}) we decompose the sum into a ``small $\ell$''
and a ``large $\ell$'' contribution. 
It is natural to consider a cut-off size separating the two sums that scales like 
$\tau$ as a function of $j$, which means that the cut-off value should behave as $\sim j-|\xi|\sqrt{j}.$
So let us define 
\begin{equation}\label{eq:jstar}
	j^{\star}:=(j-n)-(1+|\xi|)\sqrt{j-n} ,
\end{equation}
 and write the expression in (\ref{jlimit})
as
\beq
(j \Delta_j)^{\frac{n-1}{2n}}\!e^{-j\Delta_j}\left(\sum_{0\leq\ell\leq j^{\star}}\frac{(j\Delta_j)^{\ell}}{\ell !(j-\ell)^{\mu}}
+\sum_{j^{\star}<\ell\leq j-n}\frac{(j\Delta_j)^{\ell}}{\ell !(j-\ell)^{\mu}}\right) =: S_3(j)+S_4(j)    \nn
\eeq

\medskip

The ``small $\ell$'' sum, $S_3$, that corresponds to the contribution of {\it large\/} cluster
in the initial data (remember the change of variable $k\mapsto\ell$), can be estimated in the same way as 
 the sum $S_1$ in the proof of Theorem~5.1 of \cite{crw}, and was already done in 
section~6.2 of \cite{crw} with $n=2$. The general $n$ case is no different:
\beqa
S_3(j) & = & (j \Delta_j)^{\frac{n-1}{2n}}\;e^{-j\Delta_j}\sum_{0\leq\ell\leq j^{\star}}\frac{(j\Delta_j)^{\ell}}{\ell!(j-j^{\star})^{\mu}} \nn \\
& = & j^{\frac{n-1}{2n}-\frac{\mu}{2}}\left(\frac{n}{\sqrt{j}}+(1+|\xi|)\sqrt{1-\frac{n}{j}}\right)^{-\mu}
\!\!e^{-j\Delta_j}\sum_{0\leq\ell\leq j^{\star}}\frac{(j\Delta_j)^{\ell}}{\ell!}\nn \\
& \leq & C j^{\frac{n-1}{2n}-\frac{\mu}{2}}e^{-j\Delta_j}\sum_{\ell=0}^{\infty}\frac{(j\Delta_j)^{\ell}}{\ell!} \;=\;
C j^{\frac{n-1}{2n}-\frac{\mu}{2}}.\label{s3}
\eeqa
The existence of a positive constant $C$ is due to the fact that
\[
 \left(\frac{n}{\sqrt{j}}+(1+|\xi|)\sqrt{1-\frac{n}{j}}\right)^{-\mu} \longrightarrow (1+|\xi|)^{-\mu} \in\Rb^+,\;\text{as $j\to\infty$.}
\]

\medskip

We now consider the ``large $\ell$'' sum, $S_4$, corresponding to the contribution of {\it small\/} clusters.
Recalling Stirling's expansion $\ell ! = \sqrt{2\pi}\ell^{\ell + 
\frac{1}{2}}e^{-\ell}\bigl(1+\frac{1}{12\ell} + O(\ell^{-2})\bigr),$
we can write
\begin{eqnarray}
 S_4(j) & = & (j \Delta_j)^{\frac{n-1}{2n}}\!e^{-j\Delta_j}\sum_{\ell =j^{\star}+1}^{j-n}\frac{(j\Delta_j)^{\ell}}{\ell !(j-\ell)^{\mu}}\nonumber \\
        & = & (j \Delta_j)^{\frac{n-1}{2n}}\!e^{-j\Delta_j}\frac{1}{\sqrt{2\pi}}
               \sum_{\ell =j^{\star}+1}^{j-n}\frac{(j\Delta_j)^{\ell}}{\ell^{\ell + \frac{1}{2}}e^{-\ell}\bigl(1+ O(\ell^{-1})\bigr)(j-\ell)^{\mu}}\nonumber \\
        & = & \frac{1}{\sqrt{2\pi}}\bigl(1+ O(j^{-1})\bigr)(j\Delta_j)^{\frac{n-1}{2n}}\!\sum_{\ell =j^{\star}+1}^{j-n}
               \Bigl(\frac{j\Delta_j}{\ell}e^{1-\frac{j\Delta_j}{\ell}}\Bigr)^{\ell}\frac{1}{\ell^{1/2}(j-\ell)^\mu}.\label{est1}
\end{eqnarray}
Now, observing that the function $x\mapsto xe^{1-x},$ for positive $x,$ is positive and has an absolute maximum at $x=1$ with value equal to $1$; 
recalling \eqref{limitdelta}; and using the expression of $j^\star$ in \eqref{eq:jstar} to conclude that, when $j^\star +1 \leq \ell \leq j-n$, we
have $C_1j^{-1/2} \leq \ell^{-1/2}\leq C_2j^{-1/2}$, with $C_1, C_2\to 1$ as $j\to\infty,$ we can estimate \eqref{est1} by
\begin{eqnarray}
 S_4(j) & \leq & \frac{1}{\sqrt{2\pi}}\bigl(1+ O(j^{-1})\bigr)j^{\frac{n-1}{2n}-\frac{1}{2}}\sum_{\ell =j^{\star}+1}^{j-n}\frac{1}{(j-\ell)^\mu}.\label{est2}
\end{eqnarray}
To estimate the sum in the right-hand side observe that
\[
 \sum_{\ell =j^{\star}+1}^{j-n}\frac{1}{(j-\ell)^\mu} = \sum_{k=n}^{j-j^{\star}-1}\frac{1}{k^\mu},
\]
and the last sum is the lower Darboux sum of the integral of $x\mapsto x^{-\mu}$ between $n-1$ and $j-j^\star-1$. Thus, denoting by $C$ constants independent of $j$ that can vary from case to case, we have: 
when $\mu\in (0,1),$
\begin{eqnarray*}
 \sum_{k=n}^{j-j^{\star}-1}\frac{1}{k^\mu} & \leq & \int_{n-1}^{j-j^\star-1}\frac{1}{x^{\mu}}dx \\
 & = & \frac{1}{1-\mu}\Bigl((j-j^\star-1)^{1-\mu} - (n-1)^{1-\mu}\Bigr)\\
 & = & \frac{1}{1-\mu}\Bigl(\bigl((n-1)+(1+|\xi|)\sqrt{j-n}\bigr)^{1-\mu} - (n-1)^{1-\mu}\Bigr)\\
 & = & \frac{(1+|\xi|)^{1-\mu}}{1-\mu}\Bigl(1+ o\bigl(j^{-\frac{1-\mu}{2}}\bigr)\Bigr)\left(1-\frac{n}{j}\right)^{\frac{1-\mu}{2}}j^{\frac{1-\mu}{2}}\\
 & \leq & Cj^{\frac{1-\mu}{2}},
\end{eqnarray*}
when $\mu=1$,
\[
\sum_{k=n}^{j-j^{\star}-1}\frac{1}{k} \leq \int_{n-1}^{j-j^\star-1}\frac{1}{x}dx = \log\frac{j-j^\star-1}{n-1} 
\leq C\log j,
\]
and when $\mu>1$,
\[
\sum_{k=n}^{j-j^{\star}-1}\frac{1}{k^\mu} \leq \int_{n-1}^{j-j^\star-1}\frac{1}{x^{\mu}}dx
=  \frac{1}{\mu-1}\Bigl((n-1)^{1-\mu}-(j-j^\star-1)^{1-\mu}\Bigr)
\leq C.
\]
Hence \eqref{est2} can be estimated as
\begin{eqnarray*}
 \eqref{est2} \leq \begin{cases}
Cj^{-\frac{1}{2n}+\frac{1-\mu}{2}},&\quad\text{if $\mu \in (0,1)$}\\
Cj^{-\frac{1}{2n}}\log j,&\quad\text{if $\mu =1$}\\
Cj^{-\frac{1}{2n}},&\quad\text{if $\mu >1$}.
\end{cases}
\end{eqnarray*}

\medskip

Together with the estimate \eqref{s3} we conclude that, for $\mu > 1-\frac{1}{n}$,
the limit  \eqref{criticallimit} is equal to zero and the rate of convergence to the limit is
the same as for $S_4$ above, namely, 
\begin{eqnarray}
\frac{\sqrt{2\pi}}{\alpha}\left(\frac{\alpha}{n}\right)^{\frac{1}{n}}\tau^{\frac{n-1}{2n}}\mathfrak{I}_1(j,\tau) 
 \leq \begin{cases}
Cj^{-\frac{1}{2n}+\frac{1-\mu}{2}},&\quad\text{if $\mu \in \bigl(1-\frac{1}{n},1\bigr)$}\\
Cj^{-\frac{1}{2n}}\log j,&\quad\text{if $\mu =1$}\\
Cj^{-\frac{1}{2n}},&\quad\text{if $\mu >1$}.
\end{cases}                                 \label{finalsec3}
\end{eqnarray}


\section{Rate of convergence to $\Phi_{2,n}$: contribution of $\mathfrak{I}_2$}\label{sec:rateest}

We now consider the limit
\beq
\lim_{\substack{j,\,\tau \rightarrow +\infty\\ \xi=\frac{j-\tau}{\sqrt{\tau}}\; {\rm fixed}}}
\frac{\sqrt{2\pi}}{\alpha}\left(\frac{\alpha}{n}\right)^{\frac{1}{n}}\tau^{\frac{n-1}{2n}}\mathfrak{I}_2(j,\tau)\label{criticallimit2}
\eeq
with $\mathfrak{I}_2(j,\tau)$ defined by \eqref{I2}. Analogously to what was done in \cite{cps,crw}, instead of
using \eqref{I2} with $j\in \Nb$
we consider the function  $\mathfrak{I}_2(x,\tau)$ defined for $(x, \t) \in [n, \infty)\times [0, \infty)$ by
the same expression as \eqref{I2} and thus will consider instead the limit of
\begin{equation}
\varphi(x, \tau) = \frac{\frac{\sqrt{2\pi}}{\alpha}\left(\frac{\alpha}{n}\right)^{1/n}\t^{\frac{n-1}{2n}}}{\Gamma(n-x+1)}
\int_0^{\t}(\widetilde{c_1}(\tau-s))^{n-1}s^{x-n}e^{-s}\,ds \label{fixtau}
\end{equation}
when $x, \t \to +\infty$ with $\xi=\frac{x-\tau}{\sqrt{\tau}}$ fixed. Hence, substituting $x=\tau+\xi\sqrt{\tau}$
in \eqref{fixtau}, the problem is reduced to the study of the limit of
$\varphi(\tau+\xi\sqrt{\tau}, \tau)$ as $\t\to +\infty.$

\medskip

It is not difficult to see that a change of the integration variable and a few 
algebraic manipulations allow us to write \eqref{fixtau} in the following form

\begin{eqnarray}
\lefteqn{\varphi(\tau+\xi\sqrt{\tau}, \tau) =} \label{e5} \\
&=& P(\xi,\tau)
\int_0^{\tau^{1/2n}}\bigl(1 + f_n(w^n\sqrt{\tau})\bigr)\bigl(1+g(w,\tau,\xi)\bigr)
e^{-\xi w^n-\frac{w^{2n}}{2}}dw\nonumber \\
&=&P(\xi,\tau)\sum_{j=1}^4 J_j(\xi, \tau)\;,\nonumber 
\end{eqnarray}
where,
$$
P(\xi,\tau):=\frac{\sqrt{2\pi}e^{-\tau}\tau^{\tau+\xi\sqrt{\tau}-n+\frac{1}{2}}}
{\Gamma(\tau+\xi\sqrt{\tau}-n+1)}\,,
$$
\begin{equation}\label{fndef}
f_n(\tau):=-1+\left(\frac{n\tau}{\alpha}\right)^{(n-1)/n}\!\!(\tilde{c}_1(\tau))^{n-1}\,,
\end{equation}
$$
g(w,\tau,\xi):=-1+\left( 1-\frac{w^n}{\sqrt{\tau}}\right)^{\tau+\xi\sqrt{\tau}-n}\!\!\exp
\left(w^n\sqrt{\tau}+\xi w^n+\frac{w^{2n}}{2}\right)\,,
$$
and the $J_j$ are the contributions corresponding 
to taking the following in each of the bracketed terms in the integrand of \eqref{e5}:
\begin{center}
\vspace*{3mm}
\begin{tabular}{@{}|c|cc|}\hline
Term  & Contribution & Contribution \\
      & from $(1+f_n)$ & from $(1+g)$   \\ \hline
$J_1$ & $1$   & $1$  \\    
$J_2$ & $f_n$ & $1$  \\ 
$J_3$ & $1$   & $g$  \\  
$J_4$ & $f_n$ & $g$  \\  \hline 
\end{tabular}
\vspace*{3mm}
\end{center}

\medskip

We shall now estimate the behaviour as $\t\to +\infty$ of each of these terms $P$ and $J_j$, paying
attention not only to the limit but actually to the rate at which that limit is approached.

\medskip

We start by estimating the prefactor $P$:
%
%
\begin{lemma}\label{P}
$P(\xi,\tau)=e^{-\frac{\xi^2}{2}}(1+O(\tau^{-1/2}))$, as $\tau\to +\infty$.
\end{lemma}


\begin{proof}
Define $y=\tau+\xi\sqrt{\tau}$, so that,
$$
P=\frac{\sqrt{2\pi}e^{-\tau}\tau^{y-n+\frac{1}{2}}}{\Gamma(y-n+1)}\,.
$$
Since by Stirling expansion we have, as $y\to+\infty$,
\begin{align*}
\frac{1}{\Gamma(y-n+1)}&=\frac{1}{\Gamma(y)}(y-1)(y-2)\dots (y-(n-1))\\
&=\frac{y^{n-1}}{\sqrt{2\pi}e^{-y}y^{y-\frac{1}{2}}}
\frac{1+O(y^{-1})}{1+O(y^{-1})}\\
&=\frac{1}{\sqrt{2\pi}e^{-y}y^{y-n+\frac{1}{2}}}(1+O(y^{-1})),
\end{align*}
we obtain, as $\tau\to+\infty$,
\begin{align}
\nonumber
P&=e^{y-\tau}\left(\frac{\tau}{y}\right)^{y-n+\frac{1}{2}}(1+O(y^{-1}))\\
\nonumber
&=e^{\xi\sqrt{\tau}}\left(\frac{\tau}{\tau+\xi\sqrt{\tau}}\right)^{\tau+\xi\sqrt{\tau}-n+\frac{1}{2}}(1+O(\tau^{-1/2}))\\
&=e^{\xi\sqrt{\tau}}\left(\frac{\tau}{\tau+\xi\sqrt{\tau}}\right)^{\tau+\xi\sqrt{\tau}}(1+O(\tau^{-1/2}))\,.
\label{Peq1}
\end{align}
 By further observing that, 
\begin{eqnarray*}
\lefteqn{e^{\xi\sqrt{\tau}}\left(\frac{\tau}{\tau+\xi\sqrt{\tau}}\right)^{\tau+\xi\sqrt{\tau}} =}\\
&=&\exp \left[\xi\sqrt{\tau}+(\tau+\xi\sqrt{\tau})\log\left(1+\frac{-\xi\sqrt{\tau}}{\tau+\xi\sqrt{\tau}}\right)\right]\\
&=&\exp \left[\xi\sqrt{\tau}+(\tau+\xi\sqrt{\tau})\left(\frac{-\xi\sqrt{\tau}}{\tau+\xi\sqrt{\tau}}
-\frac{1}{2}\left(\frac{-\xi\sqrt{\tau}}{\tau+\xi\sqrt{\tau}}\right)^2+O(\tau^{-3/2})\right)\right]\\
&=&\exp \left[-\frac{1}{2}\xi^2\frac{\tau}{\tau+\xi\sqrt{\tau}}
+O(\tau^{-1/2})\right]\\
&=&\exp \left[-\frac{1}{2}\xi^2\left(1+O(\tau^{-1/2}\right)\right]\\
&=&e^{-\frac{\xi^2}{2}}(1+O(\tau^{-1/2}))\,,
\end{eqnarray*}
as $\tau\to+\infty$, and using this in \eqref{Peq1} the lemma is proved.
\end{proof}
%
%
\begin{lemma}\label{J1}
$J_1=\displaystyle\int_0^{+\infty}e^{-\xi w^n-\frac{w^{2n}}{2}}dw+O(e^{-\gamma \tau})$, as $\tau\to +\infty$, where $\gamma>0$ only depends on $n$ and $\xi$.
\end{lemma}


\begin{proof}
It is obvious that, as $\tau\to +\infty,$
\[
 J_1(\xi,\tau) = \int_0^{\tau^{1/2n}}e^{-\xi w^n-\frac{w^{2n}}{2}}dw \longrightarrow\int_0^{+\infty}e^{-\xi w^n-\frac{w^{2n}}{2}}dw.
\]
The issue here is to estimate at what rate the limit is approached. Since the integrand function is positive
we have 
\beq
 \left|\int_0^{\tau^{1/2n}}e^{-\xi w^n-\frac{w^{2n}}{2}}dw - 
 \int_0^{+\infty}e^{-\xi w^n-\frac{w^{2n}}{2}}dw\right| 
 = \int_{\tau^{1/2n}}^{+\infty}e^{-\xi w^n-\frac{w^{2n}}{2}}dw.\label{J1estimate}
\eeq
Start by considering $\xi\geq 0:$ 
\begin{eqnarray}
 \int_{\tau^{1/2n}}^{+\infty}e^{-\xi w^n-\frac{w^{2n}}{2}}dw 
 & < & \int_{\tau^{1/2n}}^{+\infty}e^{-\frac{w^{2n}}{2}}dw \; = \; 
       \tfrac{1}{2n}\int_{\tau}^{+\infty}z^{\frac{1}{2n}-1}e^{-z/2}dz \nonumber \\
 & < & \tfrac{1}{2n}\int_{\tau}^{+\infty}e^{-z/2}dz \;= \; \tfrac{1}{n}e^{-\tau/2}.\label{xi>0}
\end{eqnarray}
Now consider the case $\xi <0:$
\begin{eqnarray}
 \int_{\tau^{1/2n}}^{+\infty}e^{-\xi w^n-\frac{w^{2n}}{2}}dw 
 & = & \int_{\tau^{1/2n}}^{+\infty}e^{|\xi| w^n-\frac{w^{2n}}{2}}dw \; = \;
       \int_{\tau^{1/2n}}^{+\infty}e^{-\frac{w^{2n}}{2}\left(1-\frac{2|\xi|}{w^n}\right)}dw\nonumber \\
 & < & \frac{1}{2n\gamma}e^{-\gamma(\xi)\tau},      \label{xi<0}
\end{eqnarray}
where the last inequality is due to
\[
 1-\frac{2|\xi|}{w^n} > 1 - \frac{2|\xi|}{\tau^{1/2}} > 1 - \frac{2|\xi|}{\tau_0^{1/2}} =: 2\gamma(\xi),
\]
valid for all $\tau \in [\tau_0, +\infty)$, for every fixed $\tau_0>0.$ Observing that $2\gamma(\xi) <1$ we
can  use \eqref{xi>0} and \eqref{xi<0} in \eqref{J1estimate} to finally obtain
\beq
\left|\int_0^{\tau^{1/2n}}e^{-\xi w^n-\frac{w^{2n}}{2}}dw - 
 \int_0^{+\infty}e^{-\xi w^n-\frac{w^{2n}}{2}}dw\right| < \frac{1}{2n\gamma}e^{-\gamma(\xi)\tau}. \label{J1conv}
\eeq
(Observe that, due to \eqref{xi>0}, the convergence is uniform for $\xi\geq 0$.)
\end{proof}

%
%
\begin{lemma}\label{J3}
$J_3=O(\tau^{-1/2}),$ as $\tau\to +\infty.$
\end{lemma}


\begin{proof}
We perform the change of variable $x=w^{n}$, and call $\theta=\sqrt{\tau}$, so that,
\begin{equation}\label{J4xtheta}
J_3=\int_{0}^{\theta}\frac{\hat{g}(x,\theta,\xi)}{nx^{1-1/n}}\;e^{-\xi x - \frac{x^2}{2}}\,dx,
\end{equation}
where,
\begin{align*}
\hat{g}(x,\theta,\xi)&=-1+\left(1-\frac{x}{\theta}\right)^{\theta^2+\xi\theta-n}
\exp\left(x\theta+\xi x +\frac{x^2}{2}\right)\\
&=-1+\exp\left((\theta^2+\xi\theta-n)\log\left(1-\frac{x}{\theta}\right)
+x\theta+\xi x +\frac{x^2}{2}\right)\,.
\end{align*}
It is straightforward to obtain the following three developments, for $0\leq \frac{x}{\theta}<1$:
\begin{align*}
\theta^2\log\left(1-\frac{x}{\theta}\right)+x\theta+\frac{x^2}{2}&=
-\frac{x^3}{\theta}
\sum_{k=0}^{\infty}\frac{1}{k+3}\left(\frac{x}{\theta}\right)^k\,,\\
\xi\theta\log\left(1-\frac{x}{\theta}\right)+\xi x&=
-\frac{\xi x^2}{\theta}
\sum_{k=0}^{\infty}\frac{1}{k+2}\left(\frac{x}{\theta}\right)^k\,,\\
-n\theta\log\left(1-\frac{x}{\theta}\right)&=
\frac{n x}{\theta}
\sum_{k=0}^{\infty}\frac{1}{k+1}\left(\frac{x}{\theta}\right)^k\,.\\
\end{align*}
Thus, by collecting equal powers of $x/\theta$ in the above series, we obtain, for $0\leq x<\theta$,
\begin{equation}\label{qestim}
(\theta^2+\xi\theta-n)\log\left(1-\frac{x}{\theta}\right)
+x\theta+\xi x +\frac{x^2}{2}
=-\frac{x}{\theta}\sum_{k=0}^{\infty}p_k(x,\xi)\left(\frac{x}{\theta}\right)^k\,,
\end{equation}
where
$$
p_k(x,\xi):=\frac{x^2}{k+3}+\frac{\xi x}{k+2}-\frac{n}{k+1}\,.
$$
For $k=1,2,\dots,$ let $x_k(\xi)$ be the largest of the two zeros of $p_k(\cdot,\xi)$, that is,
$$
x_k(\xi)=-\frac{k+3}{2(k+2)}\xi + \sqrt{\frac{(k+3)^2}{4(k+2)^2}\xi^2 + \frac{k+3}{k+1}n}\,,
$$
so that, for $x\geq x_k(\xi)$, we have, $p_k(x,\xi)\geq 0$. By observing that,
as $k\to \infty$, $x_k(\xi)\to -\frac{\xi}{2}+\sqrt{\frac{\xi^2}{4}+n}$, we can conclude that
$x_*(\xi):=\sup x_k(\xi)<+\infty$, and hence,
$$
x\geq x_*(\xi)\quad\Rightarrow\quad \forall k\in\mathbb{N},\; p_k(x,\xi)\geq 0\,.
$$
Therefore, regardless of the sign of $\xi$,
$$
x_*(\xi)\leq x<\theta \quad\Rightarrow\quad 
(\theta^2+\xi\theta-n)\log\left(1-\frac{x}{\theta}\right)
+x\theta+\xi x +\frac{x^2}{2}\leq 0\,,
$$
which in turn implies that,
\begin{equation}\label{xstarg1}
x_*(\xi)\leq x <\theta\quad\Rightarrow\quad |\hat{g}(x,\theta,\xi)|\leq 1\,.
\end{equation}
Now we deal with the exponential term in \eqref{J4xtheta}. If $\xi\geq 0$, then obviously,
for $x\geq 0,$
$e^{-\xi x -\frac{x^2}{2}}\leq e^{-\frac{x^2}{2}}.$ If $\xi <0$, fix any $\omega \in \left]0,1\right[,$ and
define $x_{**}:=\frac{2|\xi|}{1-\omega}.$ It is straightforward to check that,
\begin{equation}\label{xstarstar}
x\geq x_{**}\quad\Rightarrow\quad e^{-\xi x-\frac{x^2}{2}}\leq e^{-\omega\frac{x^2}{2}}\,.
\end{equation}
If $\xi\geq 0$ we take $x_{**}=0$ and $\omega =1$.
Now, pick some $\eta\in\left]0,1\right[$, and consider the splitting $J_3=J_{3,1}+J_{3,2},$ where,
$$
J_{3,1}:=
\int_{0}^{\eta\theta}\frac{\hat{g}(x,\theta,\xi)}{nx^{1-1/n}}\;e^{-\xi x - \frac{x^2}{2}}\,dx\,\quad \text{  and }\quad
J_{3,2}:=
\int_{\eta\theta}^{\theta}\frac{\hat{g}(x,\theta,\xi)}{nx^{1-1/n}}\;e^{-\xi x - \frac{x^2}{2}}\,dx\,.
$$
Consider that $\theta\geq\frac{1}{\eta}\max (x_{*},x_{**})$. Due to \eqref{xstarg1} and \eqref{xstarstar},
$$
|J_{3,2}|\leq
\int_{\eta\theta}^{\theta}\frac{dx}{nx^{1-1/n}}\;e^{- \omega\frac{x^2}{2}}
\leq (1-\eta^{1/n})\theta^{1/n}e^{- \frac{\omega\eta^2}{2}\theta^2},
$$
and thus, if $\omega_*\in \left]0,\frac{\omega\eta^2}{2}\right[,$ we have
\begin{equation}\label{finalJ32}
|J_{3,2}|=o(e^{-\omega_*\theta^2})\,,\quad\text{as}\quad \theta\to +\infty\,.
\end{equation}
Now we estimate $|J_{3,1}|.$ Let $q$ be the left-hand side of \eqref{qestim}, so that,
$|\hat{g}|=|1-e^q|$. Suppose first that $q\leq 0.$ Then, $|\hat{g}|\leq -q,$ which implies, by taking \eqref{qestim} in account,
$$
|\hat{g}(x,\theta,\xi)|\leq \frac{x}{\theta}\left(\frac{x^2}{3}+\frac{|\xi|x}{2}\right)
\frac{1}{1-\frac{x}{\theta}}\leq C_1(x+x^2)\frac{x}{\theta},
$$
for all $0\leq x\leq \eta\theta$, for some positive constant $C_1$ only depending on $\xi$ and $\eta$. 

\medskip

Now, suppose that $q>0.$ We have to find lower estimates for each polynomial $p_k(\cdot,\xi)$.
If $\xi\geq 0$, then the minimum of $p_k(x,\cdot,\xi)$ for $x\in [0,\theta]$ is $p_k(0,\xi)=-\frac{n}{k+1}$. Then,
$$
q\leq \frac{x}{\theta}\sum_{k=0}^\infty\frac{n}{k+1}\left(\frac{x}{\theta}\right)^k\leq
C_2\frac{x}{\theta}\,,
$$
for all $0\leq x\leq \eta\theta$, and for some positive constant $C_2$ only depending on $n$ and $\eta$.
If $\xi<0$, then the minimum of $p_k(x,\xi)$ for $x\in [0,\theta]$ is
$m_k:=p_k\left(\frac{k+3}{2k+4}|\xi|,\xi\right)=-\frac{k+3}{k(k+2)^2}\xi^2-\frac{1}{k+1}n$\,. Then,
$$
q\leq \frac{x}{\theta}\sum_{k=0}^\infty |m_k|\left(\frac{x}{\theta}\right)^k\leq
C_3\frac{x}{\theta}\,,
$$
for all $0\leq x \leq \eta\theta$, for some positive constant $C_3$, only depending on $n,\xi$ and $\eta$. 

For both cases (i) and (ii) we thus obtain, with $C=C_j$, $j=2,3$, accordingly,
$$
|\hat{g}(x,\theta,\xi)|=e^q-1\leq \exp\left(\frac{Cx}{\theta}\right)-1=
\frac{Cx}{\theta}\sum_{k=0}^\infty \frac{1}{(k+1)!}\left(\frac{Cx}{\theta}\right)^k\leq
C_4\frac{x}{\theta}\,,
$$
for all $0\leq x\leq \eta\theta$, for some positive constant $C_4$, only depending on
$n, \xi$ and $\eta$.

\medskip

Taking in account all the previous estimates, we conclude that, regardless the 
sign of $\xi$, there is a positive constant $C_5$ only depending on $n,\xi$ and 
$\eta$ such that, for all $0\leq x\leq \eta\theta$, 
$$
|\hat{g}(x,\theta,\xi)|\leq\frac{C_5}{\theta}(x+x^2+x^3)\,.
$$
Therefore,
\begin{align}\label{finalJ31}\nonumber
|J_{3,1}|&\leq \frac{C_5}{n\theta}\int_0^{\eta\theta}(x^{1/n}+x^{1+1/n}+x^{2+1/n})
e^{-\xi x - \frac{x^2}{2}}dx\\
\nonumber
&\leq  \frac{C_5}{n\theta}\int_0^{+\infty}(x^{1/n}+x^{1+1/n}+x^{2+1/n})
e^{-\xi x - \frac{x^2}{2}}dx\\
&\leq \frac{M}{\theta}\,,
\end{align}
where $M>0$ only depends on $\xi, n$ and $\eta$.
By taking together \eqref{finalJ32} and \eqref{finalJ31}, we obtain $|J_3|=O(\theta^{-1})$, as $\theta\to+\infty,$ and considering the definition of $\theta$, the lemma is proved.
\end{proof}
%
%
\begin{lemma}\label{J4}
$J_4(\xi,\tau)=O(\tau^{-1/2})$, as $\tau\to +\infty$.
\end{lemma}


\begin{proof}
By the continuity of $f_n$ in $[0,+\infty[$ and Theorem \ref{teo:longtime}, we know that $f_n$ is bounded so that there is $M>0$, which is independent of $w$ and $\tau$ such that, for all $w\geq 0$ and $\tau\geq 0$,
$|f_n(w^n\sqrt{\tau})g(w,\tau,\xi)|\leq M |g(w,\tau,\xi)|$, and therefore we can take the estimates from the proof of lemma \ref{J3} to obtain the result.
\end{proof}

\medskip

For the next result let $\displaystyle\nu_0:=\sum_{k=n}^{\infty}c_k(0).$ 
%
%
\begin{lemma} The following holds true for $\tau$ sufficiently large:
\begin{equation}
\left|n\T^{\frac{1}{2n}}J_ 2+\nu_0\left(\frac{n}{\alpha}\right)^{(n-1)/n}\right| \leq C\T^{-\frac{1}{2}+\frac{1}{2n}}\log\T,
\qquad \text{as $\T\to\infty$}.
\end{equation}
\end{lemma}


\begin{proof}
We estimate
$$
J_2=\int_0^{\tau^{1/2n}}f_n(w^n\sqrt{\tau})e^{-\xi w^n-\frac{w^{2n}}{2}}dw\,.
$$
By performing the change of variable $w=\frac{x^{1/n}}{\tau^{1/2n}}$, we obtain,
\begin{equation}\label{J2x}
J_2=\frac{1}{n\tau^{1/2n}}\int_0^{\tau}\frac{f_n(x)}{x^{1-1/n}}e^{-\xi \frac{x}{\sqrt{\tau}}-\frac{x^{2}}{2\tau}}dx=\frac{1}{n\tau^{1/2n}}(I_1(\xi,\tau)+I_2(\xi,\tau))\,,
\end{equation}
where,
\begin{align*}
I_1(\xi,\tau)&=\int_0^{\tau}\frac{f_n(x)}{x^{1-1/n}}dx\\
I_2(\xi,\tau)&=\int_0^{\tau}\frac{f_n(x)}{x^{1-1/n}}(e^{-\xi \frac{x}{\sqrt{\tau}}-\frac{x^{2}}{2\tau}}-1)dx\,.
\end{align*}
By \eqref{fndef},
\begin{align}
I_1&=\int_{0}^{\tau}\left[\left(\frac{n}{\alpha}\right)^{1-\frac{1}{n}}(\tilde{c}_1(x))^{n-1}-x^{\frac{1}{n}-1}\right]\,dx\nonumber\\
&=\left(\frac{n}{\alpha}\right)^{1-\frac{1}{n}}\int_{0}^{\tau}(\tilde{c}_1(x))^{n-1}dx-n\tau^{\frac{1}{n}}\,.\label{J2I1}
\end{align}
 On the other hand, from \eqref{eq:xy} we have the equation,
$$
\frac{d}{d\tau}\sum_{k=n}^{\infty}\tilde{c}_k(\tau)=(\tilde{c}_1(\tau))^{n-1},
$$
from which we obtain, for each $\tau>0,$
$$
\int_0^{\tau}(\tilde{c}_1(x))^{n-1}\,dx = -\nu_0 + \sum_{k=n}^{\infty}\tilde{c}_k(\tau)\,.
$$
Now, we estimate the asymptotics of the sum in the second member of this equality.
First, we observe that, from Theorem~\ref{teo:longtime},
$$
\tilde{c}_1(\tau)=\left(\frac{\alpha}{n\tau}\right)^{1/n}\left(1+(1-\tfrac{1}{n})\frac{\log\tau}{\tau}
+ o\left(\frac{\log\tau}{\tau}\right)\right)\,,
$$
as $\tau\to+\infty.$ But, by the center manifold estimate (3.6) in Proposition 3 of 
\cite{cps}, and considering the definitions of the variables therein, we obtain, as $\tau\to+\infty,$
$$
\tilde{c}_1(\tau)\sum_{k=n}^{\infty}\tilde{c}_k(\tau) = 
\alpha - n(\tilde{c}_1(\tau))^{n+2} + \frac{n(n-1)}{\alpha^2}(\tilde{c}_1(\tau))^{2n+2} + 
O\left((\tilde{c}_1(\tau))^{2n+4}\right)\,.
$$
From the two previous expressions we deduce that, as $\tau\to\infty,$
$$
\sum_{k=n}^{\infty}\tilde{c}_k(\tau)=
\alpha\left(\frac{n}{\alpha}\right)^{1/n}\tau^{1/n} - \alpha\left(\frac{n}{\alpha}\right)^{1/n}
(1-\tfrac{1}{n})\tau^{-1+1/n}\log\tau + O\left(\tau^{-1-1/n}\right)\,.
$$
By using this in \eqref{J2I1},
\begin{align*}
I_1(\xi,\tau)&=-\nu_0\left(\frac{n}{\alpha}\right)^{1-\frac{1}{n}}+\left(\frac{n}{\alpha}\right)^{1-\frac{1}{n}}\sum_{k=n}^{\infty}\tilde{c}_k(\tau)-n\tau^{\frac{1}{n}}\\
&= -\nu_0\left(\frac{n}{\alpha}\right)^{(n-1)/n}\!\! - (n-1)\tau^{-1+\frac{1}{n}}\log\tau + 
O\left(\tau^{-1-1/n}\right)\,.
\end{align*}
Now, it remains to estimate 
\[
I_2(\xi,\tau)=\int_0^{\tau}\frac{f_n(x)}{x^{1-1/n}}\bigl(e^{-\xi \frac{x}{\sqrt{\tau}}-\frac{x^{2}}{2\tau}}-1\bigr)dx.
\]
Let us start by observing that $1-e^{-s} \leq \min\{1, s\}$. To apply this to the exponential in $I_2$ observe
that $s_+:= -\xi + \sqrt{\xi^2+2}$ is the value of $s = \frac{x}{\sqrt{\tau}}>0$ for which $\xi s + \frac{1}{2}s^2 = 1$. 
Thus, we can start by estimating $I_2$ as follows:
\begin{equation}
|I_2(\xi,\tau)| \leq \int_0^{\sqrt{\tau}s_+}\left|\frac{f_n(x)}{x^{1-1/n}}\right|
\left(|\xi| \frac{x}{\sqrt{\tau}}+\frac{x^{2}}{2\tau}\right)dx +
\int_{\sqrt{\tau}s_+}^{\tau}\left|\frac{f_n(x)}{x^{1-1/n}}\right|dx. \label{bound1I_2}
\end{equation}
Now, since in the second integral in \eqref{bound1I_2} we have $x> \sqrt{\tau}s_+$, if $\tau$ is sufficiently large
we can substitute the asymptotic form of $f_n(x)$ obtained from using \eqref{longtime} in \eqref{fndef}, in order to write
\begin{eqnarray}
\int_{\sqrt{\tau}s_+}^{\tau}\left|\frac{f_n(x)}{x^{1-1/n}}\right|dx & = & 
\int_{\sqrt{\tau}s_+}^{\tau}\Bigl((n-1)\left(1-\textstyle{\frac{1}{n}}\right)\frac{\log x}{x^{2-1/n}} + 
o\bigl(\textstyle{\frac{\log x}{x^{2-1/n}}}\bigr)\Bigr)dx   \label{applyf1} \\
& \leq & C\T^{-\frac{1}{2}+\frac{1}{2n}}\int_{1}^{\sqrt{\tau}/s_+}\textstyle{\frac{\log u}{u^{2-1/n}}}du ,\nonumber
\end{eqnarray}

\noindent
where $C$ is a constant depending on $n$ and $\xi$ (through $s_+$) and from now on can vary from line to line.
Using  the monotonicity of the logarithm to write $\log u < \log \sqrt{\tau}/s_+$ we can continue the estimate above as
\begin{eqnarray}
& \leq & C\T^{-\frac{1}{2}+\frac{1}{2n}}\log \tau\int_{1}^{\sqrt{\tau}/s_+}\textstyle{\frac{1}{u^{2-1/n}}}du \nonumber \\
& \leq & C\T^{-\frac{1}{2}+\frac{1}{2n}}\log \tau \,. \label{secondintegralbound}
\end{eqnarray}

\medskip

Now let us consider the first integral in \eqref{bound1I_2}. It is convenient to write it as an integral
between fixed limits:
\begin{eqnarray}
\lefteqn{\int_0^{\sqrt{\tau}s_+}\left|\frac{f_n(x)}{x^{1-1/n}}\right|
\left(|\xi| \frac{x}{\sqrt{\tau}}+\frac{x^{2}}{2\tau}\right)dx =} \nonumber \\
 & = & \int_0^1\left|\frac{f_n(u\sqrt{\tau}s_+)}{(u\sqrt{\tau}s_+)^{1-1/n}}\right|
\left(|\xi| s_+u+\frac{1}{2}s_+^2u^2\right)\sqrt{\T}s_+du.  \label{firstintegralfixed}
\end{eqnarray}

\medskip

In order to apply the asymptotic form of $f_n$, as in \eqref{applyf1} we need to decompose the integral in
\eqref{firstintegralfixed} into a sum of two integrals such that in one we can apply the result about 
the asymptotic behaviour of $f_n$ and the other can be adequately controlled. 

Fix $\widetilde{\varepsilon}, \delta>0$. In \eqref{firstintegralfixed} 
write $\int_0^1 = \int_0^{\delta/\T^{-\widetilde{\varepsilon}+1/2}} + 
\int_{\delta/\T^{-\widetilde{\varepsilon}+1/2}}^1.$
Then we have

\begin{eqnarray}
\lefteqn{\int_0^{\delta/\T^{-\widetilde{\varepsilon}+1/2}}\left|\frac{f_n(u\sqrt{\tau}s_+)}{(u\sqrt{\tau}s_+)^{1-1/n}}\right|
\left(|\xi| s_+u+\frac{1}{2}s_+^2u^2\right)\sqrt{\T}s_+du = } \nonumber \\
& = & (\sqrt{\tau}s_+)^{\frac{1}{n}}\int_0^{\delta/\T^{-\widetilde{\varepsilon}+1/2}}\left|f_n(u\sqrt{\tau}s_+)\right|
\left(|\xi| s_++\tfrac{1}{2}s_+^2u\right)u^{\frac{1}{n}}du \nonumber \\
& \leq & C \tau^{\frac{1}{2n}}\int_0^{\delta/\T^{-\widetilde{\varepsilon}+1/2}}\left(|\xi| s_+u^\frac{1}{n}
+\tfrac{1}{2}s_+^2u^{1+\frac{1}{n}}\right)du \label{Est2} \\
& \leq & C\tau^{-\frac{1}{2}+\varepsilon}, \label{Est3}
\end{eqnarray}
with $\varepsilon = (1+\frac{1}{n})\widetilde{\varepsilon}$,
where \eqref{Est2} is obtained by using the boundedness of $f_n$, due to its definition \eqref{fndef},
and \eqref{Est3} is valid since we have $\T$  with $\frac{\delta}{\T^{-\widetilde{\varepsilon}+1/2}}<1$,
otherwise the decomposition into two integrals would not make sense.
Furthermore, again denoting by the same letter $C$ constants that can be different, we have

\begin{eqnarray}
\lefteqn{\int_{\delta/\T^{-\widetilde{\varepsilon}+1/2}}^1\left|\frac{f_n(u\sqrt{\tau}s_+)}{(u\sqrt{\tau}s_+)^{1-1/n}}\right|
\left(|\xi| s_+u+\frac{1}{2}s_+^2u^2\right)\sqrt{\T}s_+du \leq } \nonumber \\
& \leq & C\int_{\delta/\T^{-\widetilde{\varepsilon}+1/2}}^1\left(\frac{|\log(u\sqrt{\T}s_+)|}
{(u\sqrt{\T}s_+)^{2-\frac{1}{n}}}\right)\left(|\xi| s_+u+\frac{1}{2}s_+^2u^2\right)\sqrt{\T}s_+du \nonumber \\
& \leq & C \T^{-\frac{1}{2}+\frac{1}{2n}}\int_{\delta/\T^{-\widetilde{\varepsilon}+1/2}}^1\frac{|\log u|}{u^{2-\frac{1}{n}}}
\left(|\xi| s_+u+\frac{1}{2}s_+^2u^2\right)du + \label{p1} \\
& & \quad + C\T^{-\frac{1}{2}+\frac{1}{2n}}\log \T
\int_{\delta/\T^{-\widetilde{\varepsilon}+1/2}}^1\frac{1}{u^{2-\frac{1}{n}}}
\left(|\xi| s_+u+\frac{1}{2}s_+^2u^2\right)du. \label{p2}
\end{eqnarray}
Since both  integrals
\[
\int_0^1\frac{|\log u|}{u^{2-\frac{1}{n}}}
\left(|\xi| s_+u+\frac{1}{2}s_+^2u^2\right)du
\qquad\text{and}\qquad
\int_0^1\frac{1}{u^{2-\frac{1}{n}}}
\left(|\xi| s_+u+\frac{1}{2}s_+^2u^2\right)du
\]
are convergent, we conclude from \eqref{p1} and \eqref{p2} that for $\tau$ sufficiently large we have
\begin{equation}
\int_{\delta/\T^{-\widetilde{\varepsilon}+1/2}}^1\left|\frac{f_n(u\sqrt{\tau}s_+)}{(u\sqrt{\tau}s_+)^{1-1/n}}\right|
\left(|\xi| s_+u+\frac{1}{2}s_+^2u^2\right)\sqrt{\T}s_+du \leq C\T^{-\frac{1}{2}+\frac{1}{2n}}\log \T . \label{Est4}
\end{equation}
Finally, from \eqref{secondintegralbound}, \eqref{Est3}, and \eqref{Est4} we conclude that for $\tau$ sufficiently large
we have
\[|I_2(\xi,\T)| \leq C\T^{-\frac{1}{2}+\frac{1}{2n}}\log\tau .\] 
Finally, we can  conclude the estimate of the behaviour of $J_2$ proceeding from \eqref{J2x}.
%
%
\end{proof}

%
%
\begin{theorem} \label{final}
With the notations and hypothesis introduced above in this paper we have for $\tau$ sufficiently large that 
\begin{eqnarray*}
\displaystyle\left|\tau^{1/2n}
\left(\varphi(\tau+\xi\sqrt{\tau},\tau)-\Phi_{2,n}(\xi)\right) + 
e^{-\xi^2/2}\left(\frac{\alpha}{n}\right)^{1/n}\frac{\nu_0}{\alpha}\right|
& \leq &
 C\T^{-\frac{1}{2}+\frac{1}{2n}}\log \T.
\end{eqnarray*}
\end{theorem}

\medskip

\begin{proof}
The result follows directly from the lemmata \ref{J1}--\ref{J4}. 
\end{proof}

\medskip

If there are no clusters initially present on the crystal facet, $c_j(0)=0$ for $j\geq n$, and thus $\nu_0=0.$
In this case, since we know that $\t = j\Delta_j$ with $\Delta_j = 1 + o(1)$ as $j\to\infty$ (see section~\ref{sec:rateestinit}),
we immediately conclude that Theorem~\ref{final} reduces to Theorem~\ref{teo:monomer}.


\section{Conclusion}\label{sec:conclusion}

To conclude the proof of Theorem~\ref{teo:main} it is enough to put together the result given in
Theorem~\ref{final} with the estimate \eqref{finalsec3} proved in section~\ref{sec:rateestinit} for 
the contribution of the initial condition. Under the assumption of Theorem~\ref{teo:main} that the 
initial condition $c_j(0)$ decays like $j^{-\mu}$, we conclude that the slowest decaying term 
arising from $\varphi$ decays like $j^{-\frac{1}{2n}}$, while by \eqref{finalsec3} we see that
the contribution coming from the initial data is in all cases never faster than this decay. Thus,
the contribution coming from the initial condition is 
the dominant one for the decay rate, which proves Theorem~\ref{teo:main}.

\medskip

We end by observing that Theorem~\ref{teo:main} implies that when the tail of the initial condition decays sufficiently slow
at infinity, with a rate $\mu \in \bigl(1-\frac{1}{n}, 1\bigr)$,  the rate of convergence to the similarity profile of 
the solution of the deposition model \eqref{system} will depend on $\mu$. This lack of forgetfulness of the
rate of convergence to the scaling profiles about some features of the initial condition
was also observed recently in \cite{cps} for this deposition model
along noncritical directions $\eta \neq 1$. So, although the similarity profile itself lacks any information
about the initial condition, the rate at which it is approached by the solutions of \eqref{system} can still
convey some information.


\bibliographystyle{amsplain}

\end{document}